\documentclass[11pt]{amsart}
\pdfoutput=1
\usepackage{amsmath, amssymb}
\usepackage{amsfonts}
\usepackage{graphicx}
\usepackage{lpic}
\usepackage{amsthm}
\usepackage{float}
\usepackage{amsthm}
\usepackage{mathtools}
\usepackage[top=1in, bottom=1in, left=1.5in, right=1.5in]{geometry}
\usepackage{etoolbox}
\patchcmd{\quote}{\rightmargin}{\leftmargin 2em \rightmargin}{}{}
\allowdisplaybreaks

\usepackage{hyperref}
\hypersetup{pdfauthor={},pdftitle={determining the subgraphs of the curve graph},colorlinks=true,linkcolor=black,citecolor=black}

\newtheorem{theorem}{Theorem}[section]
\newtheorem{lemma}[theorem]{Lemma}
\newtheorem{fact}[theorem]{Fact}
\newtheorem{proposition}[theorem]{Proposition}
\newtheorem{corollary}[theorem]{Corollary}

\newtheorem*{theorem*}{Theorem}
\theoremstyle{remark}

\numberwithin{equation}{section}

\newcommand{\Z}{\mathbb{Z}}

\newcommand{\N}{\mathbb{N}}
\newcommand{\Mod}{\mathrm{Mod}}
\newcommand{\R}{\mathbb{R}}
\newcommand{\Q}{\mathbb{Q}}

\newcommand{\cC}{\mathcal{C}}
\newcommand{\cAC}{\mathcal{AC}}
\newcommand{\acre}{\cAC^{re}}
\newcommand{\cP}{\mathcal{P}}

\begin{document}
\title[determining the finite subgraphs of curve graphs]{Determining the finite subgraphs\\ of curve graphs}
\author[Aougab, Biringer, Gaster]{Tarik Aougab,  Ian Biringer and Jonah Gaster}
\date{\today}

\begin{abstract}

We prove that there is an algorithm to determine if a given finite graph is an induced subgraph of a given curve graph.

\end{abstract}

\maketitle


Fix $g,n \in \Z_{\ge0}$, $(g,n)\neq (0,1)$ and let $S=S_{g,n}$ be a surface of genus $g$ with $n$ boundary components.  
Let $\cC(S)$ be the \emph{curve graph} of $S$, whose vertices are isotopy classes of nonperipheral, essential simple closed curves, 
and whose edges connect pairs of curves with the minimum possible intersection number. Unless $S$ is a torus, a once-punctured torus or a $4$-punctured sphere, edges reflect disjointness, while in those cases, the minimum intersection numbers are $1$, $1$, and $2$, respectively.  

Let $\cP_{g,n}$ indicate the collection of finite induced subgraphs of $\cC(S)$.
Our main theorem is the following:

\begin{theorem}
\label{main thm}
There is an algorithm that determines, given a graph $G$ and a pair $(g,n)$, whether or not $G\in \cP_{g,n}$.
In particular, each set $\cP_{g,n}$ is recursive.
\end{theorem}

The curve graphs of the torus, the punctured torus and the four holed sphere are all isomorphic to the \emph{Farey graph}, whose  vertex set is $\Q \cup \{\infty\}$ and where vertices $\frac ab, \frac pq$  are adjacent if $\mathrm{det} \left (\begin{smallmatrix}
	a & p \\ b & q 
\end{smallmatrix}\right )=\pm1.$
In communication with Edgar Bering, the third author has determined that a finite graph is an induced subgraph of the Farey graph if and only if its connected components are \emph{triangulated outerplanar graphs}, i.e.\ graphs that can be realized in the plane with all vertices adjacent to the unbounded face, and where all other faces are triangles. However, in general no simple  characterization of the induced subgraphs of $\mathcal C(S)$ is known.

 Note that in the statement of Theorem \ref{main thm} the pair $(g,n)$ is not fixed.  In principle, it is harder to produce an algorithm as above that takes $(g,n)$ as input than to produce a different algorithm for every $(g,n)$. For fixed $n$, the two problems are equivalent since every graph $G$ on at most $N$ vertices embeds in the curve graph of any surface of genus at least $N^{2}$ (see Remark $1.3$ of \cite{BeringGaster} and the preceding discussion). On the other hand, as we let the number of punctures grow to infinity, we do not know a quick reduction of the main theorem to the simpler problem of producing an algorithm on each surface independently. That said, in the proof of Theorem \ref{main thm} we give, no extra effort is required for the stronger statement.
 

By work of Koberda \cite{koberda2012right},   whenever a finite graph $G$  embeds  as an induced subgraph of $C(S)$, the \emph{right angled Artin group} (RAAG) $\Gamma(G)$ whose defining graph is $G$ embeds  as a subgroup of the mapping class group $\Mod(S)$.  The converse is not true in general: if $\Gamma(G)$  embeds in $\Mod(S)$, one can only say that $G$ embeds in the \emph{clique graph} of $C(S)$ as an induced subgraph \cite{kim2014geometry}. However, in \cite{kim2013right} Kim--Koberda show that for surfaces with $3g-3+n<3$, any embedding of $\Gamma(G)$ into $\Mod(S)$ does give an embedding of $G $ into $C(S)$  as an  induced subgraph. Hence,

\begin {corollary}
When $S$ is a sphere with at most $5$ punctures, or a torus with at most $2$ punctures, there is an algorithm that determines whether or not a given RAAG embeds in $\Mod(S)$.

\end {corollary}

After reading the above, one might ask whether there is an algorithm to detect whether a given RAAG embeds in another given RAAG, since Kim--Koberda \cite{kim2014geometry} have shown that this is equivalent to the defining graph of the first embedding in the \emph{extension graph}  of the second, an analogue of the curve graph. Indeed, such an algorithm has been recently given by Casals-Ruiz \cite{casals2015embeddability}, and her proof is somewhat similar in spirit to ours, although in our view it is more complicated.

The graph $C(S)$ is highly symmetric---its  quotient under the action of  the mapping class group $\Mod(S)$ has diameter $1$. If $C(S)$ were locally finite, which it is not, one could use this symmetry to give a trivial proof of Theorem \ref{main thm}, at least for fixed $S$. Namely, one could pick any vertex $v \in C(S)$, and given a graph $G$ on $N$  vertices, just check through all the finitely many subgraphs of the ball $B_{C(S)}(v,N+1)$ to see if the connected components of $G$ all appear.

Our proof of Theorem \ref{main thm} is a refinement of this naive approach. We will show:

\begin{quote}
	($\star$) \it If a graph $G$ on $N$  vertices embeds as an  induced subgraph of $C(S)$, it also embeds as a  system of curves on $S$ whose total self-intersection number is bounded by a computable function of $N$.\rm
\end{quote}
Up to the action of $\Mod(S)$, curve systems with  bounded self-intersection number can be enumerated, say by checking through bounded complexity triangulations of $S$ for curve systems embedded in their $1$-skeleta. So, to check whether a given graph $G$ embeds, we can just  compare $G$ to each such curve system. 

 We prove ($\star$) by induction on the complexity of $S$.  To facilitate this, it is useful to replace $C(S)$  with the \emph{arc and curve graph rel endpoints}, $\acre(S)$. We discuss this graph in \S \ref{MMW},  where we give a Masur--Minsky type summation formula for intersection numbers in $\acre(S)$, using a similar result of Watanabe \cite{Watanabe} that applies to intersection numbers of simple closed curves. Finally, in \S \ref{pfsec}  we prove a stronger technical version of ($\star$) for $\acre(S)$, which finishes the proof of Theorem \ref{main thm}.

\subsection{Acknowledgements}  The first author was supported by NSF grant DMS-- 1502623, and the second author was supported by NSF grant DMS--1611851. Thanks to Edgar Bering, Thomas Koberda, Sang-Hyun Kim and Johanna Mangahas for helpful conversations.

\section{Intersection numbers of arcs rel endpoints}

\label{MMW}

If $S$ is a compact, orientable surface, the \emph{arc and curve graph rel endpoints} of $S$ is the graph $\cAC^{re}(S)$  whose vertices are isotopy classes of non-peripheral, essential simple closed curves (i.e.~`curves'), and isotopy classes rel endpoints of properly embedded, essential arcs (i.e.~`arcs').  Edges connect pairs of vertices that 
intersect minimally among arcs and curves on $S$.

 Note that this  is
not the familiar \emph{arc and curve graph} $\cAC(S)$: when $S$ is closed or is an annulus, then $\cAC(S) := \acre(S)$ as defined above, but otherwise $\cAC(S)$ is a quotient of $\acre(S)$ where arcs are identified if they are properly isotopic. 

\vspace{2mm}

In \cite{Watanabe}, Watanabe shows  that the intersection number of a pair of curves on a surface $S$  can be estimated by summing up the distances between their subsurface projections, a la Masur--Minsky \cite{Masurgeometry2}. His theorem extends and effectivizes previous work of Choi-Rafi \cite{Choicomparison}. Here, we show that his theorem also applies to intersection numbers in $\acre(S)$, if projections to peripheral annuli are added to the sum.


\begin{theorem}\label{mm}There  is a computable function  $C=C(k,\mathfrak X)$ as follows. Given a pair of  vertices $\alpha,\beta$ in $\acre(S)$ and a constant $k>0$, we have
\[
\log \iota( \alpha ,  \beta) \underset{C}{\asymp} \left(  
\sum_{Y\subseteq S}  \left[[ d_Y( \alpha, \beta) \right]]_k +
\sum_{A\subset S} \log  \left[[ d_A( \alpha, \beta) \right]]_k  \right),
\] 
where $Y,A$ above are isotopy classes of subsurfaces of $S$ that intersect both $\alpha$ and $\beta$  essentially, the first sum is taken over non-annular $Y$ and the second sum is taken over (possibly peripheral) annuli $A$. Here, $f \underset{C}{\asymp} \, g$ if $f \leq C g +C \text{ and } g \leq Cf+C.$
\end{theorem}

 To understand the statement, recall that if $Y\subset S$  is a nonannular subsurface the \emph{subsurface projection} $\pi_Y : \acre(S) \longrightarrow \cAC(Y)$ takes a vertex $\alpha$  to its intersection with $Y$, say after $Y$ and $\alpha$ are put in minimal position. Note that the co-domain of $\pi_Y$ is not $\acre(Y)$, since there is no canonical embedding of $Y\subset S$.
 A different definition is needed for $\pi_A$ when $A$ is an annulus, since then $\cAC(S)=\acre(S)$ and arcs are considered up to isotopy rel endpoints. Here, the cover $S_A$ of $S$ corresponding to $\pi_1 A$  compactifies to an annulus, and we let $\pi_A(\alpha)$ be  any lift of $\alpha$ to this annulus that  connects its two boundary components, see \cite{Masurgeometry2}.  One then defines $$d_Y(\alpha,\beta):=d_{\cAC(Y)}(\pi_Y(\alpha),\pi_Y(\beta)), \ \ d_A(\alpha,\beta):=d_{\cAC(S_A)}(\pi_A(\alpha),\pi_A(\beta)),$$  Also, by definition $[[n]]_k=0$ if $n<k$ and $[[n]]=n$  otherwise, and the $\log$ appearing in the statement of the theorem is a modified version of the logarithm so that $\log(0)= 0$. 
  
If $A$ is a peripheral annulus, the definition of $\pi_{A}$ is the same as in the non-peripheral case, the only difference being that the cover corresponding to $\pi_{1}A$ will have one geodesic boundary component and thus only needs to be compactified on one end. We stress that here $\alpha,\beta$  are vertices of $\acre(S)$, so $\iota(\alpha,\beta) $ means the minimum number of intersections between arcs/curves \emph{isotopic rel endpoints} to $\alpha,\beta$.  It is for this reason that peripheral annuli are included in the summation.

The following lemma is the heart of Theorem \ref{mm},  and will also be used in the proof of the  Proposition \ref {technical lemma}.

\begin{lemma}[Good annuli] \label{intersecting with the annulus}  Assume $S $ is a compact, orientable surface  that is not an annulus, and $\gamma \subset S$  is a (potentially peripheral) simple closed curve. Then there is a hyperbolic metric on $S$ with geodesic boundary, and a metric neighborhood $A$ of the geodesic representative of $\gamma$, such that 
\begin {enumerate}
\item  any simple geodesic $\alpha$ in $S$ is in minimal position with respect to $\partial A$, so that its intersection with $A$  is  a subset of the vertices of $\acre(A)$.
\item[(2)] for any two simple geodesics $\alpha,\beta$ on $S$, the projection distance $d_{A}(\alpha,\beta)$ is within $2 $ of $d_{\acre(A)}(\alpha \cap A,\beta \cap A)$.
\end{enumerate}
 Moreover,  if $\partial S = \gamma_1 \cup \cdots \cup \gamma_n$, then there is a hyperbolic metric and a collection of associated annuli $A_1,\ldots, A_n$ satisfying (1) and (2) such that 
 \begin{enumerate}
\item[(3)] $\iota_{\mathcal {AC}(S)}(\alpha,\beta)$ is within $2 $ of $|\alpha  \cap \beta \cap (S \setminus \cup_i A_i)|$.
\end {enumerate}
\end{lemma}

 Recall that $d_A(\alpha,\beta) :=d_{\acre(S_A)}(\pi_A(\alpha),\pi_A(\beta))$, where $S_A$ is the annular cover of $S$ corresponding to $\pi_1 A$, and $\pi_A $ is defined by taking an appropriate lift. So, the point of (2) is that up to an additive error, one can also define the projection distance to an annulus by intersecting. In (3), $\iota_{\mathcal {AC}(S)}(\alpha,\beta)$  is  the usual intersection number in $\mathcal {AC}(S)$.  That is, it is the minimum number of intersections that we can realize using arcs properly homotopic to $\alpha,\beta$, but where the homotopy can move the endpoints.

\begin{proof} 
 Pick a hyperbolic metric on $S$ with geodesic boundary such that $\gamma $ is a geodesic with length one. (If $\gamma$ is peripheral, first homotope it to be a boundary component of $S$.)  By the Collar Lemma \cite[Lem 13.6]{Farbprimer},  if we set $$r=\sinh^{-1}\left ( \frac 1{\sinh(\frac 12)} \right ),$$ then  the metric $r$-neighborhood of $\gamma$ is an embedded annulus $A$. This radius $r$ has the following (related) property. Regard $\gamma$ as the quotient of a geodesic $\tilde \gamma$ in $\mathbb H^2$ by an isometry $g : \mathbb H^2 \longrightarrow \mathbb H^2$ stabilizing $\tilde \gamma$, and let $\tilde A$ be the $r$-neighborhood of $\tilde \gamma$ in $\mathbb H^2$. 
The value of $r$ is chosen so that whenever $x\in \partial_\infty \mathbb H^2$ is not one of the  endpoints of $\tilde \gamma$, the  hyperbolic geodesic joining $x,g(x)$ is tangent to $\tilde A$ (see Figure \ref{annulusPic}(a)): 
Briefly, conjugate $g$ so that $g(z)=ez$ and so that $x=1$, and consider the right triangle in Figure \ref{annulusPic}(a) with base angle $\theta$, whose hypotenuse lies on the $x$-axis. A simple calculation in $\mathbb{H}^2$ gives
$$r = \sinh^{-1}(\cot \theta) = \sinh^{-1} \frac {\sqrt{\left (\frac {e+1}2 \right )^2  - \left (\frac {e- 1}2\right )^2}}{(e-1)/2} = \sinh^{-1}\left ( \frac 1{\sinh(\frac 12)} \right ).$$
 It follows that $A$ satisfies (1).  Indeed, any  geodesic $\alpha$ in $S$ that intersects some component of $\partial A$ nonminimally has a lift $\tilde \alpha$ that intersects $\partial \tilde A \subset \mathbb H^2$ nonminimally. 
This implies $g(\tilde \alpha)\cap \tilde \alpha \neq \emptyset$, so $\alpha$  cannot be simple (see Figure \ref{annulusPic}(a)).  

\begin{figure}[t]
\centering
\includegraphics{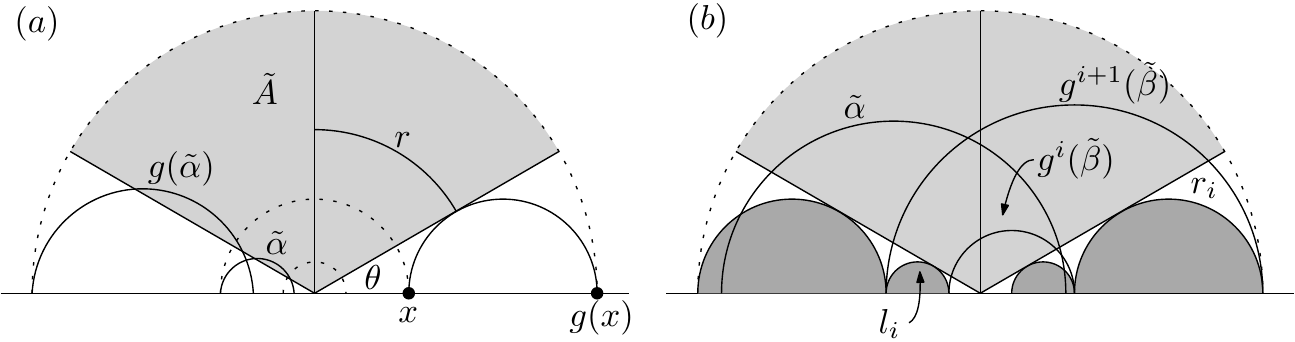}
\caption{}
\label{annulusPic} 
\end{figure}

To prove (2), let $\alpha,\beta$  be simple geodesics on $S$ and $\tilde \alpha,\tilde \beta$ be lifts of $\alpha,\beta$ in $\mathbb H^2$. Since the distance in the arc and curve graph between two arcs in an annulus is their geometric intersection number, we have that
$$d_A(\alpha,\beta) =d_{\mathcal{AC}(S_A)}(\pi_A(\alpha),\pi_A(\beta))= | \{ i \in \Z \ | \ g^i(\tilde \beta) \cap \tilde \alpha \neq \emptyset \}|.$$ 
And similarly, we have from (1) that   $$d_{\acre(A)}(\alpha \cap A,\beta \cap A) = | \{ i \in \Z \ | \ g^i(\tilde \beta) \cap \tilde \alpha \cap \tilde A \neq \emptyset \}|.$$
So, to prove (2) it suffices to show that at most two of the $g^i(\tilde \beta)$  can intersect $\tilde \alpha$ outside of $\tilde A$. But this follows immediately from Figure \ref{annulusPic}(b). Namely, for each $i$ draw the geodesics $l_i,r_i$ joining the left (resp.\ right) endpoint of $g^i(\tilde \beta)$ to that of $g^{i+1}(\tilde \beta)$. Then each $l_i,r_i$ is tangent to $\partial \tilde A$, and bounds a half-plane outside of $\partial \tilde A$, shaded in dark grey in Figure \ref{annulusPic} (b). The geodesic $\tilde \alpha$ is then the union of at most five segments: one in $\partial \tilde A$, at most two in shaded half planes,  and at most two in the remaining white regions. Since $\tilde \alpha$  can intersect at most one iterate $g^i(\tilde \beta)$ per white component, (2) follows.

For property $(3)$ above, choose similarly a hyperbolic metric on $S$ so that each of the boundary components $\gamma_1,\ldots,\gamma_n$ are geodesics with length $1$, and let $A_1,\ldots,A_n$ be their $r$-neighborhoods on $S$. The Collar Lemma implies that all the $A_i$ are disjoint, in addition to being embedded annuli; see the paragraph after the proof of \cite[Lem 13.6]{Farbprimer}.  As before, fix simple geodesic arcs $\alpha,\beta$ in $S$.

The universal cover of $S$  is a convex subset $\tilde S \subset \mathbb{H}^2$ bounded by a collection of bi-infinite geodesics.  The intersection number $\iota(\alpha,\beta)$ can be computed by fixing a lift $\tilde \alpha$ of $\alpha$, and counting the number of lifts of $\beta$ that intersect it. Let's suppose $\alpha$ has  endpoints on boundary components $\gamma_a,\gamma_b$ of $S$, with possibly $a=b$, let $\tilde \gamma_a,\tilde \gamma_b$ be the boundary components of $\tilde S$ incident to $\tilde \alpha$, and let $\tilde A_a,\tilde A_b$ be the metric $r$-neighborhoods of these $\tilde \gamma_a,\tilde \gamma_b$ in $\tilde S$. Then
\begin{align}\label{firsteq} |\alpha  \cap \beta \cap (S \setminus \cup_i A_i)| &= \big |\big\{\text{lifts } \tilde \beta \text{ of } \beta \ \big | \ \tilde \beta \cap \tilde \alpha \cap (\tilde S \setminus (\tilde A_a \cup \tilde A_b) )\neq \emptyset \big \} \big | \\
\label{2ndeq}\iota_{\mathcal {AC}(S)}(\alpha,\beta)&=\big |\big\{\text{lifts } \tilde \beta \text{ of } \beta \ \big | \ \tilde \beta \cap \tilde \alpha \neq \emptyset \text{ and } \tilde \beta \cap (\tilde \gamma_a \cup \tilde \gamma_b) = \emptyset \big \} \big |\end{align}
For \eqref{firsteq}, note that  $\alpha $ does not enter any of the constructed annuli on $S$ other than $A_a,A_b$, nor does it enter $A_a,A_b$ except at initial and terminal segments, by property (1) established above. For \eqref{2ndeq},   note that the right side is the number of lifts $\tilde \beta$ that are incident to boundary components of $\tilde S$ that \emph{link} with $\tilde \gamma_a$ and $\tilde \gamma_b$,  meaning that they alternate with $\tilde \gamma_a$ and $\tilde \gamma_b$ in the  cyclic order on $\partial \tilde S \cup \partial_\infty \tilde S \cong S^1$. If one pinches the boundary components of $S$ to cusps, the geodesics in the homotopy classes of $\alpha,\beta$ in the resulting surface  have intersection number $\iota_{\mathcal {AC}(S)}(\alpha,\beta)$. One calculates this intersection number in $\tilde S$ by counting linked lifts, and clearly each $\tilde \beta$ from the right side of \eqref{2ndeq} contributes such a linked lift after pinching, while $\tilde \beta$ that are incident to either $\tilde \gamma_a$ or $\tilde \gamma_b$ become asymptotic to $\tilde\alpha$ after pinching.


Now  the set on the right side of \eqref{2ndeq} is \emph{contained} in the set on the right side of \eqref{firsteq}, since by property (1) of the annuli, the geodesic $\tilde \beta$ only enters $\tilde A_a,\tilde A_b$ if it is incident to either $\tilde \gamma_a$ or $\tilde \gamma_b$. So, for property (3) it suffices to show that the number of $\tilde \beta$  that are incident to either $\tilde \gamma_a$ or $\tilde \gamma_b$, but intersect $\tilde \alpha$  outside of the annuli $\tilde A_a,\tilde A_b$, is at most $2$. However, this follows immediately from the same argument we used to prove property (2) above. There is at most one such $\tilde \beta$  incident to $\tilde \gamma_a$ --- we get one instead of two since $\gamma_a$ is a boundary component of $S$, and the previous argument gave one intersection per side ---  and at most one incident to $\tilde \gamma_b$.\end{proof}

 We are now ready to prove  the desired Masur--Minsky  distance formula for intersection numbers in $\acre(S)$, using work of Watanabe \cite{Watanabe}.

\begin{proof}[Proof of Theorem \ref{mm}]
If $\alpha$ and $\beta$ are simple closed curves, then the desired inequality is the content of Theorem $1.5$ of \cite{Watanabe}.  This result can be applied more generally to estimate the intersection number of  any pair of vertices in $\mathcal{AC}(S)$, with slight changes in the constants. Indeed, there is a map $\mathcal{AC}(S)\longrightarrow \mathcal{C}(S)$ that takes an arc to any simple closed curve constructed by concatenating it with a segment (or two) of $\partial S$, and this map changes intersection numbers by at most $2$, and coarsely preserves all distances between subsurface projections.

So, suppose that $\alpha,\beta$ are arcs in $S$.  We want to estimate the intersection number $\iota(\alpha,\beta)$, where now we are not allowed to move the endpoints of $\alpha,\beta$  when we homotope them to be in minimal position. Let $A_1,\ldots,A_n$  be the peripheral annuli in Lemma \ref{intersecting with the annulus}, part (3), consider $S $ with the associated hyperbolic metric, and homotope $\alpha,\beta$ rel endpoints to be geodesics. Then 
\begin{align*}
	\iota(\alpha,\beta)  = | \alpha \cap \beta | & = |\alpha  \cap \beta \cap (S \setminus \cup_i A_i)| + \sum_{i=1}^n | \alpha \cap \beta \cap A_i |,
\end{align*}
 which by Lemma \ref{intersecting with the annulus}  is within $2+2n$ of 
 \begin{equation}\iota_{\mathcal {AC}(S)}(\alpha,\beta) +\sum_{i=1}^n d_{A_i}(\alpha,\beta).\label{sum}\end{equation}
So as $\log (x_1+\cdots + x_{n+1}) \asymp_{n+1} \log(x_1) +\cdots + \log(x_{n+1})$, it follows from \eqref{sum} and Watanabe's result for $\mathcal {AC}(S)$ that for some computable $C=C(k,\mathfrak X)$, we have
$$
\log \iota( \alpha ,  \beta) \underset{C}{\asymp} \left(  
\sum_{Y\subseteq S}  \left[[ d_Y( \alpha, \beta) \right]]_k +
\sum_{\substack{\text{nonperipheral} \\ A\subset S}} \log  \left[[ d_A( \alpha, \beta) \right]]_k  \right) + \sum_{i=1}^n \log d_{A_i}(\alpha,\beta),
$$
which proves Theorem \ref{mm} after adjusting $C$ to account for the missing $[[ \ \  ]]_k$ in the last summation. \end{proof}

\section{The proof}
\label{pfsec}

 In this section we prove the following proposition, which clearly implies $(\star)$ from the introduction, and therefore the main theorem.

\begin{proposition}
\label{technical lemma}

There is a computable function $f:\N^2\to\N$ as follows. Write $\mathfrak X= 3g-3+n$ for the  complexity of the surface $S_{g,n}$, and let $G$  be a graph whose vertices can be partitioned into  $N$ cliques. 
If $\phi : G \longrightarrow \acre(S_{g,n})$ is an embedding of $G$ into the arc and curve graph rel endpoints, then either
 $$\iota	  (\phi(\alpha),\phi(\beta)  ) \leq f(\mathfrak X,N), \ \ \forall \text{ vertices }  a,b \in G$$ or  there is another embedding $\psi : G \longrightarrow \acre(S_{g,n})$ such that 
 \begin{enumerate}
 \item $\iota (\psi(a),\psi(b)  ) \leq \iota	  (\phi(a),\phi(b)  ) $  for all vertices $a,b$ of $G$, and for \emph{some} choice of $a,b$ the inequality is strict,
 \item  for every vertex $a $ of $G$, the vertices $\phi(a),\psi(a) \in \acre(S_{g,n})$  have the same type (arc/curve) and the same endpoints on $\partial S$  if they are arcs.
\end {enumerate}
\end{proposition}

The proof will be by induction on $\mathfrak X$.  In some sense, the base case is the empty surface, in that the argument we will give below also works directly for an annulus (the  nonempty surface for which $\mathfrak X$ is minimal). But although it is not strictly necessary, we think it will be informative and comforting to the reader to start by proving the proposition  directly when $S$ is an annulus.

\begin{proof}[Proof for $\mathfrak X=-1$] Here, $\cAC^{re}(S)$ contains only one non-arc vertex, which is a connected component of $\cAC^{re}(S)$. Removing from $G$ any vertex that maps to this curve, we may assume that $\phi(a)$ is an arc for every $a \in G$.  Assuming further that $\iota(\phi(\alpha),\phi(\beta)) > 3N+1$, we'll show how to modify $\phi$  to decrease this intersection number, without increasing any other intersection number.

Choose an identification of $S$ with $[0,1]\times [0,1] / (x,0) \sim (x,1)$. Tighten each $\phi(a)$ to a Euclidean geodesic and let $s(a)$ be the resulting slope, i.e.\ the $\R$-valued number of times the arc wraps around the annulus. Let's assume without loss of generality that $s(\beta)>s(\alpha)$. Since $\iota( \alpha, \beta)$ differs by at most $1$ from $|s(\alpha) - s(\beta) |$, we have $$s(\beta)-s(\alpha)>3N$$ 
Each of the $N$ cliques of vertices in $G$ gives a set of slopes that lies in an interval of length one in $\R$. Hence, there is an interval $$[x,y] \subset [s(\alpha),s(\beta)]$$ of length $2$ in which no slopes lie. Subtracting $1$ from every slope in $[y,s(\beta)]$ then gives a new embedding of $G$ in which no intersection numbers are increased, $\iota(\alpha,\beta)$ strictly decreases, and where the endpoints of arcs on $\partial S$ are preserved.
\end{proof} 

%

The idea of the proof in the general case is as follows. Assuming that $\phi$ cannot be modified to decrease intersection numbers, we use Theorem \ref{mm} and induction to argue that all the projections of $\phi(G)$ into \emph{proper subsurfaces} have bounded diameter.  Then we perform a more complicated version of the `subtract one from all slopes' argument from the annulus case to show that the diameter of $\phi(G)$ in $\acre(S)$ itself is bounded.
 
\subsection{Proof of Proposition \ref{technical lemma}}
We proceed by induction on $\mathfrak X$.  The distrustful reader can take $\mathfrak X=-1$ as a base case, but really a trivial base case suffices. 

 To avoid excess notation, we will suppress the embedding $\phi$ in the proof of the proposition. So, let $G$ be an  induced subgraph of $\acre(S_{g,n})$ whose vertices can be partitioned into $N$ cliques. Assume that $G$ cannot be re-embedded into $\acre(S_{g,n})$ in a way that satisfies (1) and (2) in the statement of the proposition. That is, one cannot re-embed $G$ to strictly decrease the intersection number of some $\alpha$ and $\beta$ without either increasing some other intersection number or altering the endpoints of some arcs on $\partial S $.  We want to bound the intersection numbers of vertices of $G$.


If $Y\subset S$ is a proper subsurface, we will first fix a preferred realization of $Y$ coming from a hyperbolic metric as follows. If $Y$ is non-annular, choose an arbitrary hyperbolic metric and choose $Y$ to be a metric subsurface of $S$ with totally geodesic boundary. If $Y$ is annular, equip the surface with the metric coming from Lemma \ref{intersecting with the annulus} and identify $Y$ with the metric neighborhood of its core described in the statement of that lemma. 

Then realizing $G$ as a system of geodesic arcs and curves on $S$ equipped with the appropriate metric, $G$ will automatically be in minimal position with respect to $Y$. Intersecting each such arc/curve with $Y$, we obtain a graph $H$ embedded into $\acre(Y)$. Note that we have no a priori control on the number of vertices of $H$: the intersection of a vertex of $G$ with $Y$ may have arbitrarily many connected components, and we want to regard the components as distinct. However, the vertices of $H$ can be partitioned into $N$ cliques, since those of $G $ can. Now if we select  components $a,b$ of $\alpha \cap Y$ and $b$ of $\beta \cap Y$, there is no way to modify the embedding of $H$ in $\acre(Y)$ to decrease $\iota(a,b)$ without increasing the intersection numbers of other vertices of $H$ or changing endpoints on $\partial Y$, since any such modification would extend to a modification of $\phi : G \longrightarrow \acre(S)$. So, as $\mathfrak X(Y) < \mathfrak X(S)$, we have by induction that $\iota(a,b) \leq f(\mathfrak X(Y),N)$.  In particular, we have that the distance between $\alpha,\beta$ in the arc and curve complex of $Y$  satisfies
$$d_Y(\alpha,\beta)\leq 2\cdot \iota(a,b)+4 \leq  2 \cdot f(\mathfrak X(Y),N)+4,$$ 
where the multiplicative $2$ comes from the standard upper bound on distance in $\mathcal{AC}$ in terms of intersection number, and the additive $4$ comes again from this bound and from clause $(3)$ of Lemma \ref{intersecting with the annulus}. 
So, if we choose a cut-off $k$ that is larger than  $2 \cdot f(\mathfrak X(Y),N)+4$  for \emph{every}  proper subsurface $Y \subset S$, each of the summands in  Theorem \ref{mm}  that corresponds to a proper subsurface is zero. Hence, we have
by Theorem \ref{mm} that
\begin{align}
\log \iota(\alpha,\beta) &\le C \left[[ d_S ( \alpha,\beta) \right]]_k + C, \label{ibounds}
\end{align}
 where $C=C(\mathfrak X,N)$  is some computable function. 
If the diameter of $G$ in $\acre(S)$ was bounded by some computable function of $N,\mathfrak X$ then we would be done, but a priori this may not be the case.  However,  we do have the following fact about finite point sets.
  
 \begin{fact}[Small clusters, large gaps]
 Given a function $g : \N \longrightarrow \N$ and a set $G$ of $N$ points in some metric space, we can write $G$ as a union of disjoint subsets $G = \bigcup_i G_i$
  such that for some $D=D(g,N)$ we have
  \begin {enumerate}
  \item $\mathrm{diam}(G_i) \leq D$  for all $i$,
  \item	 $d(G_i,G_j) > g(D)$ for all $i\neq j$.
  \end {enumerate}\label{cluster}
 \end{fact} 
\begin{proof}[Proof of Fact \ref{cluster}]
 The proof is by induction. Start with the $G_i$ as singleton sets, and begin combining them. If the current diameter is $D$ and all the sets are $g(D)$-separated, we are done. If not, combine two close sets, replace $D$ with $g(D)+2D$ and continue. This process terminates since $N$ is finite.
\end{proof}

 Clearly, the fact also applies to our $G$, which is a union of $N$ cliques. So, leaving $g$ unspecified for the moment, let $G= \bigcup_i G_i$ and $D$ be as above.    We claim that it is possible to move each $G_i$ with a mapping class $f_i : S \longrightarrow S$ so that 
\begin {enumerate}
\item $d(f_i(G_i),f_j(G_j)) \geq 2$ if $i\neq j$,
\item all intersection numbers between vertices of $\bigcup_i f_i(G_i)$ are bounded by some constant $B=B(\mathfrak X,N,D).$
\end {enumerate}
 Here, the first condition ensures that the resulting union $\bigcup_i f_i(G_i)$ is still a subgraph of $\acre(S)$  isomorphic to $G$. Now as long as we had picked $$g : \N \longrightarrow N, \ \ g( \cdot ) > \log B(X,N,\cdot ),$$
 we will have for $a_i\in G_i$ and $a_j\in G_j$, where $i\neq j$, that $$\iota(a_i,a_j) \geq 2^{d(a_i,a_j)} > B \geq \iota(f_i(a_i),f_j(a_j))$$  so intersection numbers strictly decrease  when we replace $G$ with $\bigcup_i f_i(G_i)$.

 We define the mapping classes $f_i$ one by one, starting with $f_1 = id$. Since $G$ is a union of $N$ cliques, we can further decompose it as a union of at most $\mathfrak X \cdot N$ \emph {parallel cliques}, consisting of arc/curves that have the same projections to the usual arc and curve graph $\mathcal{AC}(S)$. In both $G_1$ and $G_2$, discard for the moment all but one representative of each parallel clique, so that the number of vertices in each is at most $\mathfrak X \cdot N$. Pick
minimal position arcs and curves representing these vertices and then separately for each $i=1,2$, extend their union on $S$ to triangulations $T_1,T_2$ of $S$ without adding new vertices.   Since the intersection number of any two  vertices in the same $G_i$ is bounded above by some computable function of $\mathfrak X,N,D$ and the number of vertices is also bounded, the total numbers of triangles in $T_1,T_2$ are also bounded in terms of $\mathfrak X,N,D$. We now use:

\begin {lemma}\label {lemma1}
There is a computable function	$K=K(n)$  such that if $T_1, T_2 $ are triangulations of a common surface $S$ that each have at most $n$ triangles, then there are refinements of $T_1,T_2$, each with at most $K$ triangles, that are combinatorially isomorphic.\end {lemma}

 Assuming the lemma, we choose $f_2$  to be the map $f$ realizing the isomorphism in the lemma. Since all the vertices of $G_1,G_2$ can be realized as closed paths with no edge repeats on $T_1,T_2$, the intersection number  between any vertices of $G_1$ and $f_2(G_2)$ is now  bounded by some function of $\mathfrak X,N,D$.  Moreover,  a similar bound (perhaps increased by $2$) will hold if we add back to $G_1$ and $f_2(G_2)$ all the previously deleted members of the parallel cliques, since if $a',b'$ are disjoint from $a,b$, respectively, and have the same projections to $\mathcal AC(S)$, then the intersection number $\iota(a',b')$ differs by at most one from $\iota(a,b)$.  The only problem is that we may not have $d(G_1,f_2(G_2)) \geq 2$ anymore, as required above. 

To remedy this, we will need the following.
 
 \begin {lemma}\label {lemma2}
 	There is a computable function	$K=K(n,M)$  such that if $T$ is a triangulation of a surface $S$  that has  at most $n$ triangles and $M > 0$, there is a homeomorphism $h : S \longrightarrow S$  that acts with translation length at least $M$  on $\acre(S)$, such that $T,h(T)$ are transverse and intersect  at most $K$ times.
 \end {lemma}

  One then applies Lemma \ref{lemma2} to $f_2(T_2)$, with $M$ equal to the sum of the $\acre(S)$-diameters of $G_1,f_2(G_2)$. The given $h$ moves $f_2(G_2)$ so that none of its vertices are adjacent to vertices of $G_1$, while keeping intersection numbers  controlled, so we can replace $f_2$ with $g \circ f_2$. This  finishes the proof of the proposition if $G=G_1\cup G_2$. If there are more than two sets in the union, we continue inductively. Combine $G_1$ and $f_2(G_2)$ into a single set, and then run the argument above to find some $f_3$ so that all intersection numbers between  vertices of $G_1 \cup f_2(G_2)$ and $f_3(G_3)$ are bounded, while using Lemma \ref{lemma2} to ensure that $f_3(G_3)$ is not adjacent to the previous two sets.  Continuing this process with $G_4,G_5,\ldots$ proves the proposition, for as the total number of the $G_i$ is at most $N$, the final bounds on intersection number will be computable in terms of $\mathfrak X,N,D$.

\vspace{2mm}

 It remains to prove the two lemmas. 
 
 \begin {proof}[Proof of Lemma \ref{lemma1}]
Via induction on the complexity of $S$, we will prove the stronger statement that \emph{if $T_1,T_2$ are triangulations of $S$ with at most $n$ triangles that agree on $\partial S$, then there are refinements of $T_1,T_2$, each with at most $K(n)$  triangles, that are combinatorially isomorphic via a map that is the identity on $\partial S$.} Note that this  implies the lemma, since one can subdivide and isotope any two triangulations so that they agree on  $\partial S$ while adding only computably many vertices.

Assume first that $S$ is a disc, which we identify with a convex polygon in $\R^2$. We claim that if $T$ is a triangulation of $S$, then after passing to a refinement whose complexity is computably bounded in terms of that of $T$, we can isotope $T$ rel $\partial S$ to be a \emph{Euclidean} triangulation of $S$.  A pair of line segments in $\R^2$ can intersect at most once, so if we apply the above to two triangulations $T_1,T_2$, the `intersection' of the resulting Euclidean triangulations will be a common refinement of $T_1,T_2$ whose complexity is computably bounded in terms of the complexities of $T_1,T_2$.

To do this, one first proves that any triangulation $T$ of a topological disc can be computably refined to be isomorphic to a Euclidean triangulation of \emph{some} convex polygon $P$.  This is done by induction on the number of triangles of $T $, and the base case is trivial. For the inductive case, take a triangulation $T$ and a triangle $\Delta \subset T$ such that $T \setminus \Delta$ is still a topological disc. (Such $\Delta$  correspond to vertices of the dual graph that don't separate.) Choose a polygonal realization $P$ of a refinement of $T \setminus \Delta$, as given by the inductive step.  The union of the two interior sides of $\Delta$ is a concatenation of at most $4 \cdot 8^{n-1}$ line segments, so after subdividing $\Delta$ into $8^n$ triangles we can append it to $P$  as indicated below.

\begin{figure}[h]
	\centering
	\includegraphics{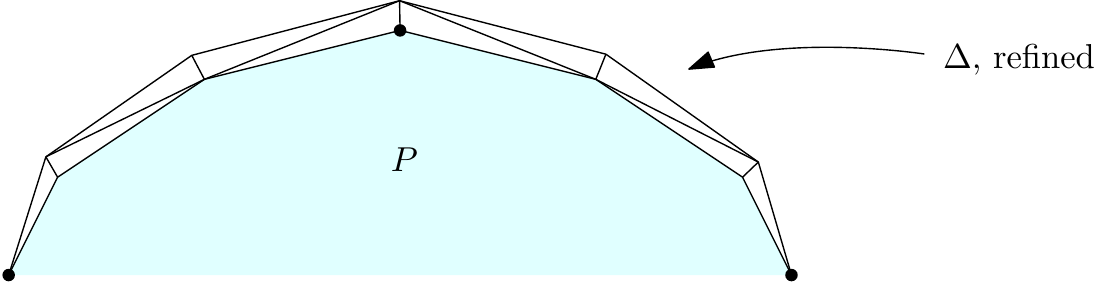}
\end{figure}

Now pick a triangulation $T$ of our convex polygon $S$. After  passing to a computable refinement of $T $, there is a combinatorial isomorphism $$f  : ( S' ,T ') \longrightarrow (S,T )$$ for some Euclidean triangulation $T '$ of a convex polygon $ S' $. As long as $T $ has been refined to include all the vertices of $S$, we may assume that $f$ is affine on each edge of $\partial S'$. Find a Euclidean triangulation $T'' $ of $ S' $ with no interior vertices that agrees with $T '$ on $\partial S'$. Then  the map $f $ can be isotoped rel $\partial S'$ to be affine on each triangle of $T'' $. The image $f (T ')$ is a triangulation of $S$ isotopic rel $\partial S$ to $T $. Since each pair of edges from $T'' $ and $T '$  intersects at most once, the edges of $f (T ')$ are piecewise linear with at most a computable number of corners. Hence, $f (T ')$ can be computably refined to a Euclidean  triangulation of $S$.

 Now suppose that $S $ is some general  compact orientable surface  and $T $  are triangulations of $S $.  Suppose first that $S $ has at most one boundary component.  If $S$  is a disc or a sphere, the conclusion follows from the base case.  Otherwise, there is a closed path $\gamma $ on  the $1$-skeleton of $T $ that is nontrivial in $H_1(\hat S)$,  where $\hat S$ is the closed surface obtained by capping off $\partial S$, if it is nonempty.  If we take $\gamma  $ to have minimal length, it cannot repeat vertices, and hence is a nonseparating   simple closed curve on $S$. After passing to a computable refinement and modifying $T_2$, say, by a homeomorphism of $S$ that is the identity on $\partial S$, we can assume that the two curves $\gamma $ are both some fixed curve $\gamma \subset S$ and the triangulations $T $ agree along $\gamma $. 

Cutting $S$ along $\gamma$, we obtain two triangulations of a new surface with lower complexity. By induction, after passing to  computable refinements there is a  combinatorial isomorphism between these triangulations that is the identity on $\partial S \cup \gamma$.  This isomorphism then glues to give an isomorphism rel $\partial S$ of $T_1$ and $T_2$.

The case when $S$ has more than one boundary component is similar, except that now instead of cutting along a non-separating simple closed curve, which may not exist, we cut along an arc connecting two distinct boundary components of $S$.
%
 \end {proof}

 \begin {proof}[Proof of Lemma \ref{lemma2}]
 	After refining $T$, we can find a pair of filling simple closed curves $\alpha,\beta$ that appear as cycles in its $1$-skeleton. (Construct a triangulation of $S$ with boundedly many vertices that includes such $\alpha,\beta$ in its $1$-skeleton and apply the previous lemma.)   Without loss of generality, we can pass to such a refinement, since the new number of vertices will be some computable function of $n$.

 Let $\tau_\alpha$ and $\tau_\beta$  be the Dehn twists  around $\alpha,\beta$,  respectively. By a theorem of Thurston \cite{Fathitravaux}, $\tau_\beta^{-1} \circ \tau_\alpha$  is a pseudo-Anosov map, so has (stable) translation distance at least some computable constant $\epsilon=\epsilon(S)>0$, by a theorem of Masur--Minsky \cite[Prop 2.1]{Masurgeometry2}. (See also Bowditch \cite{bowditch2008tight} and Gadre--Tsai \cite{gadre2011minimal} for more recent explicit bounds.) Hence, there is some computable $k$ such that $$h=(\tau_\beta^{-1} \circ \tau_\alpha)^k$$ has translation length bigger than $M$.

 It remains to bound $\iota(T,h(T))$.
If $\tau_\alpha$  is the Dehn twist around $\alpha$, then $$\iota(T,\tau_\alpha(T)) \leq 100 n^2.$$ Without looking for optimal constants, one can justify this by  noting that edges incident to $\alpha$ are twisted around it in $\tau_\alpha(T)$, so that they now intersect all of the other edges incident to $\alpha$ in the original triangulation $T$. This is the reason for the quadratic exponent. The constant $100$ is there to overcompensate for the additional intersections between $T$ and $\tau_\alpha(T)$ that one sees elsewhere in the surface after perturbing the two triangulations to be transverse.

Now, $\beta$ intersects $\tau_\alpha(T)$ at most $100n^2$ times, so a similar argument gives a computable bound for $\iota(T,\tau_\beta^{-1} \circ \tau_\alpha(T))$, which in particular gives a computable about for $\iota(\alpha,\tau_\beta^{-1} \circ \tau_\alpha(T))$.   Iterating  this process twist by twist, one obtains a computable (if terrible) bound for $\iota(T,h(T))$. \end {proof}

\bibliographystyle{hmath}
\bibliography{total}

\end{document}